\documentclass[final,3p,times]{elsarticle}
\usepackage{hyperref,color}
\usepackage{amssymb}
\usepackage{amsthm,amsmath}
\biboptions{sort&compress,square,numbers}
\usepackage{multirow}
\usepackage{enumerate}
\usepackage{graphicx}
\usepackage{subfigure}

\numberwithin{equation}{section}
\newtheorem{theorem}{Theorem}[section]
\newtheorem{proposition}[theorem]{Proposition}
\newtheorem{corollary}[theorem]{Corollary}
\newtheorem{lemma}[theorem]{Lemma}
\theoremstyle{remark}

\newtheorem{remark}[theorem]{Remark}

\newcommand{\E}{\mathcal{E}}
\newcommand{\B}{\mathcal{B}}
\newcommand{\rg}{\mathrm{Im}}

\newcommand{\R}{R_{0}}
\newcommand{\J}{\mathcal{P}}
\newcommand{\C}{G}
\newcommand{\Px}{\mathbb{P}_x}
\newcommand{\Exx}{\mathbb{E}_x}

\newcommand{\supp}{\operatorname{supp}}

\begin{document}

\begin{frontmatter}

\title{Existence of invariant densities for semiflows with jumps\tnoteref{label1}}

\tnotetext[label1]{This research was supported by the Polish NCN grant no 2014/13/B/ST1/00224}

\author{Weronika Biedrzycka\corref{cor1}}
\cortext[cor1]{Corresponding author}
\ead{wsiwek@us.edu.pl}

\author{Marta Tyran-Kami\'nska\corref{cor2}}
\address{Institute of
Mathematics, University of Silesia, Bankowa 14, 40-007 Katowice, POLAND}
\ead{mtyran@us.edu.pl}


\begin{abstract} The problem of existence and uniqueness of absolutely continuous invariant measures for a class of  piecewise deterministic Markov processes  is investigated using the theory of substochastic semigroups obtained through the Kato--Voigt perturbation theorem on the $L^1$-space.
We provide a new criterion for the existence of a strictly positive and unique invariant density for such processes. The long time qualitative behavior of the corresponding semigroups is also considered. To illustrate our general results we give a detailed study of a two dimensional model of gene expression with bursting.
\end{abstract}

\end{frontmatter}

\section{Introduction}\label{sec:int}

We study a class of piecewise-deterministic Markov processes (PDMPs) which we call semiflows with jumps.  As defined in \cite{davis84,davis93} a  PDMP   without active boundaries  is determined by three local
characteristics $(\pi,\varphi,\J )$, where $\pi$ is a semiflow describing the deterministic parts of the
process, $\varphi(x)$ is the intensity of a jump from $x$, and $\J (x,\cdot)$ is the distribution of the state
reached by that jump. The problem of existence of invariant measures for Markov processes is of fundamental importance in many applications of stochastic processes \cite{davis93,meyn93,almcmbk94}.

We consider semiflows that arise as solutions of ordinary differential equations
\begin{equation}\label{e:de}
x'(t)=g(x(t)),
\end{equation}
where $g\colon \mathbb{R}^d\to\mathbb{R}^d$ is a (locally) Lipschitz continuous mapping.
We assume that $E$ is a  Borel subset of $\mathbb{R}^d$ such that for each $x_0\in E$ the solution $x(t)$ of
\eqref{e:de} with initial condition $x(0)=x_0$ exists and that $x(t)\in E$ for all $t\ge 0$. We denote this solution
$\pi_tx_0$.  Then the mapping $(t,x_0)\mapsto\pi_tx_0$ is Borel  measurable and satisfies $\pi_0x=x$, $\pi_{t+s}x=\pi_t(\pi_sx)$ for  $x\in E$, $s,t\in\mathbb{R}_+$.   As concern jumps we consider a family of measurable transformations  $T_\theta \colon E\to E$, $\theta\in \Theta$, where $\Theta$ is a metric space  which carries a Borel measure $\nu$, and  a family of measurable functions $p_\theta\colon E\to [0,\infty)$,
$\theta\in \Theta$, satisfying
\[
\int_\Theta p_\theta(x)\nu(d\theta)=1, \quad x\in E,
\]
so that the stochastic kernel $\J $ is of the form
\begin{equation}\label{d:jumpkernel}
\J (x,B)=\int_\Theta 1_B(T_\theta (x))p_\theta(x)\nu(d\theta),\quad  x\in E,
\end{equation}
for $B\in \B(E)$, where $\mathcal{B}(E)$ be the Borel $\sigma$-algebra of subsets of $E$. This roughly means that if the value of the process is $x$ then we jump to the point $T_{\theta}(x)$ with probability $p_{\theta}(x)$.

The following standing assumptions will be made.
The intensity function $\varphi$ is continuous and
\begin{equation}\label{eq:phi2}
\lim_{t\to\infty} \int_{0}^{t}\varphi(\pi_s x)ds=+\infty\quad \text{for all  } x\in E.
\end{equation}
The  mappings  $(\theta,x)\mapsto T_\theta (x)$ and  $(\theta,x)\mapsto p_\theta (x)$ are measurable so that the stochastic kernel in \eqref{d:jumpkernel} is well defined.
We assume also that each mapping $\pi_t\colon E\to E$ as well as  each $T_\theta \colon E\to E$ is nonsingular with respect to a reference measure $m$ on $E$.  Recall that a measurable transformation $T\colon E\to E$ is called \emph{nonsingular} with respect to $m$ if the measure $m\circ T^{-1}$ is absolutely continuous with respect to $m$,
i.e., $m(T^{-1}(B))=0$ whenever  $m(B)=0$.

Let us briefly describe the construction of the PDMP $\{X(t)\}_{t\ge 0}$
with characteristics $(\pi,\varphi,\J )$ (see e.g.~\cite{davis84,davis93} for
details).
Define the function
\begin{equation}\label{eq:cdf}
F_x(t)=1-\exp\{-\int_0^t \varphi(\pi_sx)ds\}, \quad t\ge 0, x\in E,
\end{equation}
and note that the assumptions imposed on $\varphi$ imply that  $F_x$  is a distribution function of a positive
and finite random variable for every $x\in E$. Let $t_0=0$ and let $X(0)=X_0$ be an $E$-valued random variable.
For each $n\ge 1$ we can choose the $n$th \emph{jump time} $t_n$ as a positive random variable satisfying
\[
\Pr(t_n-t_{n-1}\le t|X_{n-1}=x)=F_x(t), \quad t\ge 0,
\]
and we define
\begin{equation*}
X(t)=\left\{
  \begin{array}{ll}
  \pi_{t-t_{n-1}}(X_{n-1}) & \text{ for } t_{n-1}\le t<t_{n},\\
X_n &\text{ for } t=t_n,
\end{array}\right.
\end{equation*}
where the $n$th \emph{post-jump position} $X_n$ is an $E$-valued random variable such that
\[
\Pr(X_n\in B|X(t_{n}-)=x)=\J (x,B),
\]
and $X(t_{n}-)=\lim_{t\uparrow t_n}X(t)=\pi_{t_n-t_{n-1}}(X_{n-1})$. In this way, the trajectory of the process
is defined for all $t<t_\infty:=\lim_{n\to\infty}t_n$ and $t_\infty$ is called the explosion time. To define the
process for all times, we set $ X(t) =\Delta$ for $t\ge t_\infty$, where $\Delta\notin E$ is some extra state
representing a cemetery point for the process. The PDMP $\{X(t)\}_{t\ge 0}$ is  called the \emph{minimal} PDMP corresponding to $(\pi,\varphi,\J)$. It is said to be \emph{non-explosive} if $\Px(t_{\infty}=\infty)=1$ for $m$-almost every ($m$-a.e.) $x\in E$, where $\Px$ is the distribution of the process starting at $X(0)=x$. We denote by $\Exx$ the expectation operator with respect to~ $\Px$.

Our main result is the following.
\begin{theorem}\label{t:cont}
Assume that the chain $(X(t_n))_{n\ge 0}$ has only one invariant probability measure $\mu_*$  absolutely continuous
with respect to~$m$. If the density $f_*=d\mu_{*}/dm$ is strictly positive  a.e. then the process $\{X(t)\}_{t\ge 0}$ is non-explosive and it can have at most one invariant probability measure absolutely continuous with respect to~$m$. Moreover, if
\begin{equation}\label{eq:R0v}
\int_{E} \Exx(t_1)f_*(x) m(dx) < \infty,
\end{equation}
then the process $\{X(t)\}_{t\geq0}$ has a unique invariant density and it is strictly positive a.e.
\end{theorem}

The problem of existence and uniqueness of an invariant probability measure for the process $\{X(t)\}_{t\geq0}$ with comparison to the similar problem for the chain $(X(t_n))_{n\ge 0}$ was studied in \cite{costa90} in the context of general PDMPs with boundaries and under some technical assumptions. We also refer the reader to \cite{dufourcosta99,costadufour08} for the study of
equivalence between stability properties of continuous time processes and yet another discrete time processes associated with them. Here we concentrate on the existence of absolutely continuous invariant measures and we make use of the results from \cite{tyran09}. That is why we need to assume that the semiflow $\{\pi_t\}_{t\ge 0}$ satisfies $\pi_t(E)\subseteq E$ for all $t\ge 0$ (this implies that there are no active boundaries)  and that the stochastic kernel $\J $ describing jumps gives rise to a transition operator $P$ on $L^1$ (see \eqref{d:top}) so that we can use \cite[Theorem 5.2]{tyran09}. In particular, the kernel $\J $ as in \eqref{d:jumpkernel} has the required property and covers many interesting examples. However, any refinements entail considerable mathematical difficulties and are currently under research.

We study  the continuous time process with the help of a strongly continuous semigroup of positive contraction operators  $\{P(t)\}_{t\ge 0}$  (\emph{substochastic semigroup}) on the $L^1$ space of  functions  integrable with respect to the measure $m$.
The semigroup can be obtained from the Kato--Voigt perturbation theorem for substochastic semigroups on $L^1$-spaces and this functional analytic  framework  is recalled in Section \ref{ssec:evol} as Theorem~\ref{thm:kato}.
Using results from \cite{tyran09}, this gives that the chain $(X(t_n))_{n\ge 0}$
has the property that there exists a unique linear operator $K$ (\emph{stochastic operator}) on $L^1$  which satisfies: if the distribution of the random variable $X(0)$ has a density $f$, i.e.,
\[
\Pr(X(0)\in B)=\int_B  f(x)m(dx),\quad B\in \mathcal{B}(E),
\]
then $X(t_1)$ has a density $K f$.   Hence, the density $f_*$ in Theorem~\ref{t:cont} is  invariant for the operator $K$. Sufficient conditions for the existence of only one invariant density for stochastic operators are described in Section~\ref{sec:SSG} and are based on~\cite{rudnicki95,rudnickipichortyran02}. Section \ref{ssec:evol} presents relationships between invariant densities for the semigroup $\{P(t)\}_{t\ge0}$  and  for the operator~$K$. Here the most important results are obtained in Theorems~\ref{thm:exunifp} and~\ref{thm:exunifpi} and give   Corollary~\ref{c:unique} which is our main tool in the proof of Theorem~\ref{t:cont}.
Theorems~\ref{thm:exunifp}  and~\ref{thm:exunifpi} together with Corollaries~\ref{c:exunifp} and~\ref{c:exunifpi} should be compared with \cite[Theorems 1 and 2]{costa90} and \cite[Theorem 5]{mokhtar15}. However, we need not to assume  that the process is non-explosive and we look for absolutely continuous subinvariant measures. Moreover, in \cite{mokhtar15} a perturbed substochastic semigroup is obtained with the help of Desch's theorem \cite{desch}, which in our setting becomes a particular case of Theorem~\ref{thm:kato}.

If for some $t>0$ and for $x$ from a set of positive Lebesgue measure the absolutely continuous part in the Lebesgue decomposition of the measure $\Px(X(t)\in \cdot)$ is nontrivial, then the semigroup $\{P(t)\}_{t\ge 0}$ is partially integral as in \cite{pichorrudnicki00}. This allows us to combine Theorem~\ref{t:cont}
with~\cite[Theorem 2]{pichorrudnicki00}, recalled in Section~\ref{sec:SSG} as  Theorem~\ref{thm:pr00}, to  obtain  asymptotic stability of the semigroup $\{P(t)\}_{t\ge 0}$, i.e., the density of $X(t)$ converges to the invariant density in $L^1$ irrespective of the density of $X(0)$.  In that case condition \eqref{eq:R0v} appears to be not only sufficient but also necessary for the existence of an invariant density for the process, see Corollary~\ref{cor:as_stab}.

In Section~\ref{sec:pdmp} we provide sufficient conditions for existence of a unique invariant density for the Markov chain $(X(t_n))_{n\ge 0}$ in terms of the local characteristics of the semiflow with jumps. We also show that  dynamical systems with random switching  evolving in $\mathbb{R}^d\times I$ with  a finite set $I$, as in \cite{pichorrudnicki00,bakhtinhurth12,benaim12}, can be studied with our methods. Section \ref{sec:example} contains a detailed study of a two dimensional model of gene expression with bursting illustrating applicability of our results. Our framework can be used to analyze biological processes  described by PDMPs,  see e.g.  \cite{friedman06,mackeytyran08,mcmmtkry11,mcmmtkry13} for gene regulatory dynamics with bursting and \cite{lipniacki06,rudnicki07,zeiserfranz10,tomski15,rudnickityran15} for dynamics  with switching.

\section{Asymptotic behavior of stochastic operators and semigroups}\label{sec:SSG}

Let $(E,\E,m)$ be a $\sigma$-finite
measure space and $L^1=L^1(E,\E,m)$ be the space of integrable functions. We denote by $D(m)\subset L^1$ the set
of all \emph{densities} on $E$, i.e.
\[
D(m)=\{f\in L^1_+: \|f\|=1\}, \quad \text{where } L^1_+=\{f\in L^1: f\ge 0\},
\]
and $\|\cdot\|$ is the norm in $L^1$.  A linear operator $P\colon L^1\to L^1$ such that $P(D(m))\subseteq D(m)$
is called \emph{stochastic} or \emph{Markov}~\cite{almcmbk94}. It is called \emph{substochastic} if $P$ is a
positive contraction, i.e., $Pf\ge 0$  and $\|Pf\|\le \|f\|$ for all $f\in L_+^1$.

 If $T\colon E\to E$ is nonsingular then there exists a unique stochastic operator $\widehat{T}\colon  L^1\to L^1$ satisfying
\[
\int_{B}\widehat{T}f(x) m(dx)=\int_{T^{-1}(B)}f(x) m(dx)
\]
for all $B\in\E$ and $f\in D(m)$. The operator $\widehat{T}$ is usually called \cite{almcmbk94} the \emph{Frobenius-Perron} operator corresponding to $T$. In particular, if $T\colon E\to E$ is one-to-one and nonsingular with respect to $m$, then
\[
\widehat{T}f(x)=1_{T(E)}(x)f(T^{-1}(x))\frac{d (m\circ T^{-1})}{dm}(x)\quad \text{for $m$-a.e. }x\in E,
\]
where $d(m\circ T^{-1})/dm$ is the Radon-Nikodym derivative of the measure $m\circ T^{-1}$ with respect to $m$.

Let $\J \colon E\times \E\to[0,1]$ be a \emph{stochastic transition kernel}, i.e., $\J (x,\cdot)$ is a
probability measure for each $x\in E$ and the function $x\mapsto\J (x, B)$ is measurable for each $B\in\E$, and
let $P$ be a stochastic operator on $ L^1$. If
\begin{equation}\label{d:top}
\int_E \J (x,B)f(x)m(dx)=\int_B Pf(x)m(dx)
\end{equation}
for all $B\in \E, f\in D(m)$,
then $P$ is called the \emph{transition} operator corresponding to $\J $.
A stochastic operator $P$ on $ L^1$ is called \emph{partially integral} or \emph{partially kernel} if there
exists a measurable function $p\colon E\times E\to[0,\infty)$ such that
$$
\int_E\int_E p(x,y)\,m(dx)\, m(dy) >0 \quad\text{and}\quad P f(x) \ge \int_E p(x,y)f(y)\, m(dy)
$$
for $m$-a.e.~$x\in E$ and for every density $f$.

We can extend a substochastic operator $P$ beyond the space $ L^1$  in the following way.
If $0\le f_n\le f_{n+1}$, $f_n\in L^1$, $n\in\mathbb{N}$, then the pointwise almost everywhere limit of $f_n$
exists and will be denoted by $\sup_n f_n$. For  $f\ge 0$  we define
\[
Pf=\sup_n Pf_n \quad \text{for }  f=\sup_n f_n, f_n\in L^1_+.
\]
(Note that $Pf$ is independent of the particular approximating sequence $f_n$ and that $Pf$ may be infinite.)
Moreover, if $P$ is the transition operator corresponding to $\J $ then \eqref{d:top} holds for all measurable
nonnegative $f$.  A~nonnegative measurable $f_*$ is said to be \emph{subinvariant
(invariant)} for a substochastic operator $P$ if
$Pf_*\le f_*$ ($Pf_*=f_*$). Note that if $f_*$ is a subinvariant density for a stochastic operator $P$ then $f_*$ is invariant for~$P$.

A substochastic operator $P$ is called \emph{mean ergodic} if
\begin{equation*}
\lim_{N\to\infty}\frac{1}{N}\sum_{n=0}^{N-1} P^n f \quad \text{exists for all } f\in  L^1.
\end{equation*}
If a substochastic operator has a subinvariant density $f_*$ with $f_*>0$ a.e., then it is mean ergodic
(see e.g.~\cite[Lemma 1.1 and Theorem 1.1]{kornfeldlin00}).
We say that a stochastic operator is \emph{uniquely mean ergodic} if there is an invariant density $f_*$ such that
\begin{equation}\label{d:ume}
\lim_{N\to\infty}\frac{1}{N}\sum_{n=0}^{N-1} P^n f=f_*\|f\| \quad \text{for all } f\in L^1_{+}.
\end{equation}
In particular, if $P$ has a unique invariant density $f_*$ and $f_*>0$ a.e. then $P$ is uniquely mean ergodic (see e.g. \cite[Theorem 5.2.2]{almcmbk94}).
Moreover, an operator with this property can not have a non-integrable subinvariant function as the  following result shows. For any measurable  $f$ the \emph{support} of $f$ is defined up to sets of measure $m$ zero by \[
\supp f = \{x\in E: f(x)\neq 0\}. \]

\begin{proposition}\label{p:exsub}
Suppose that a stochastic operator $P$ is uniquely mean
ergodic with an invariant density $f_*$. If $\tilde{f}_*$  is subinvariant for $P$ and  $m(\supp f_{*}\cap \{x:\tilde{f}_*(x)<\infty\})>0$,  then $\tilde{f}_*\in L^1$.
\end{proposition}
\begin{proof}
It is a direct consequence of~\eqref{d:ume} and the fact that the measure $m$ is $\sigma$-finite.
\end{proof}

To prove that an operator has a unique strictly positive invariant density we use the approach  from \cite{rudnicki95,rudnickipichortyran02}.
A stochastic operator $P$ is called \emph{sweeping} with respect to a set $B\in\E$ if
\begin{equation*}\label{sweeps1}
\lim_{n\to\infty}\int_{B} P^n f(x)m(dx)=0\quad\text{for all }f\in D(m).
\end{equation*}
 From Lemma 2 and Theorem 2 of \cite{rudnicki95} we obtain the following result
\begin{theorem}\label{t:inv}
Let $E$ be a metric space and $\E=\mathcal{B}(E)$ be the $\sigma$-algebra of Borel subsets of $E$.
Suppose that $P$ is the transition operator corresponding to the stochastic kernel $\J$ satisfying the following conditions
\begin{enumerate}[\upshape(a)]
\item\label{cond:comms} there is no $P$-\emph{absorbing sets}, i.e., there does not exist a set $B\in\E$ such that $m(B)>0$, $m(E\setminus B)>0$ and
$\mathcal{P}(x,B)\geq 1_B(x)$ for $m$-a.e. $x\in E$,
\item\label{cond:funs} for every $x_0\in E$ there exist $\delta>0$, a nonnegative measurable function $\eta$ satisfying $\int \eta(y)m(dy)>0$,
and a positive integer $n$ such that
\[
 \mathcal{P}^n(x,B) \ge 1_{B(x_0,\delta)}(x)\int_B \eta(y)m(dy)
\]
for $m$-a.e.~$x\in E$ and all $B\in \mathcal{B}(E)$, where $B(x_0,\delta)$ is the ball with center at $x_0$ and radius~$\delta$.
\end{enumerate}
Then either $P$ is sweeping with respect to compact sets or $P$ has an invariant density $f_*$.
In the latter case, $f_*$ is unique and $f_*>0$ a.e.
\end{theorem}

In order to exclude sweeping we can use a Foster--Lyapunov drift condition \cite{meyn93,rudnicki97b}. For the proof of the following see e.g. \cite{tyran03}.
\begin{proposition}\label{l:sweeping}
Let $P$ be the transition operator corresponding to a stochastic transition kernel $\mathcal{P}$. Assume that the following condition holds
\begin{enumerate}[\upshape (a)]
\item[\upshape(c)]
there exist a set $B_0$, two positive constants $c_1$, $c_2$, and a nonnegative measurable function $V$ satisfying $m(x:V(x)<\infty)>0$ and
\begin{equation}\label{cond:Lap}
\int_{E}V(y)\J(x,dy)\leq V(x)-c_1+c_21_{B_0}(x),\quad x\in E.
\end{equation}
\end{enumerate}
Then
\[
\liminf_{N\rightarrow\infty}\frac{1}{N}\sum_{n=0}^{N-1}\int_{B_0}P^nf(x)m(dx)\geq\frac{c_1}{c_2}>0
\]
for all $f\in D(m)$ such that $\int_E V(x)f(x)m(dx)<\infty$. In particular, $P$ is not sweeping with respect to the set~$B_0$.
\end{proposition}

We conclude this section with the notion of stochastic semigroups and a general result from \cite{pichorrudnicki00} concerning possible asymptotic behavior of such semigroups.
A family of substochastic (stochastic) operators $\{P(t)\}_{t\ge 0}$ on $L^1$  which is a
$C_0$-\emph{semigroup}, i.e.,
\begin{enumerate}[\upshape (1)]
\item $P(0)=I$ (the identity operator);
\item $P(t+s)=P(t)P(s)$ for every $s,t\ge 0$;
\item for each $f\in L^1$ the mapping $t\mapsto P(t)f$ is continuous: for each $s\ge 0$
\[
\lim_{t\to s^{+}}\|P(t)f-P(s)f\|=0;
\]
\end{enumerate}
is called a \emph{substochastic (stochastic) semigroup}. A nonnegative measurable $f_*$ is said to be
\emph{subinvariant (invariant)} for the semigroup $\{P(t)\}_{t\ge 0}$ if it is subinvariant (invariant) for each
operator $P(t)$.

 A
stochastic semigroup $\{P(t)\}_{t\ge 0}$ is called \emph{asymptotically stable} if it has an invariant
density $f_*$ such that
\begin{equation*}\label{d:as}
\lim_{t\to\infty}\|P(t) f-f_*\|=0\quad\text{for all }f\in D(m)
\end{equation*}
and  \emph{partially integral} if, for some $s>0$, the operator $P(s)$ is partially integral.

\begin{theorem}[\cite{pichorrudnicki00}]\label{thm:pr00}
Let  $\{P(t)\}_{t\ge 0}$ be a partially integral stochastic semigroup. Assume that the semigroup $\{P(t)\}_{t\ge
0}$ has only one invariant density $f_*$. If  $f_*>0$ a.e. then the semigroup $\{P(t)\}_{t\ge 0}$ is
asymptotically stable.
\end{theorem}

Note that if the semigroup $\{P(t)\}_{t\ge 0}$ is asymptotically stable then, for each $s>0$, the operator
$P(s)$ is uniquely mean ergodic.
Thus, Proposition~\ref{p:exsub}  gives the following.

\begin{corollary}\label{p:exsubct} Suppose that  a stochastic semigroup $\{P(t)\}_{t\ge
0}$ is asymptotically stable with an invariant density $f_*$.  If $\tilde{f}_*$  is subinvariant for $\{P(t)\}_{t\ge 0}$ and $m(\supp f_*\cap\{x:\tilde{f}_*(x)<\infty\})>0$,  then $\tilde{f}_*\in L^1$.
\end{corollary}

\section{Existence of invariant densities for perturbed semigroups}\label{ssec:evol}

In this section we study the problem of existence of invariant densities for substochastic semigroups on $L^1$.
We first recall some notation and a generalization of Kato's perturbation theorem~\cite{kato54}.

Let $\{S(t)\}_{t\ge 0}$ be a substochastic semigroup on $ L^1$. The infinitesimal \emph{generator} of
$\{S(t)\}_{t\ge0}$ is by definition the operator $A$ with domain $\mathcal{D}(A)\subset L^1$ defined as
\[
\begin{split}
\mathcal{D}(A)&=\{f\in L^1: \lim_{t\downarrow 0}\frac{1}{t}(S(t)f-f) \text{ exists} \},\\
Af&=\lim_{t\downarrow 0}\frac{1}{t}(S(t)f-f),\quad f\in \mathcal{D}(A).
\end{split}
\]
The operator $A$ is closed with $\mathcal{D}(A)$ dense in $L^1$. If for some real $\lambda$ the operator
$\lambda-A:=\lambda I-A$ is one-to-one, onto, and $(\lambda -A)^{-1}$ is a bounded linear operator, then $\lambda$
is said to belong to the resolvent set $\rho(A)$ and $R(\lambda,A):=(\lambda  -A)^{-1}$ is called the resolvent
at $\lambda$ of $A$. If $A$ is the generator of the substochastic semigroup $\{S(t)\}_{t\ge0}$ then
$(0,\infty)\subset \rho(A)$ and we have the integral representation
\[
R(\lambda,A)f=\int_{0}^{\infty}e^{-\lambda s}S(s)f\,ds \quad \text{for}\quad f\in L^1.
\]
The operator $\lambda R(\lambda,A)$ is substochastic and $R(\mu,A)f\le R(\lambda,A)f$ for $\mu>\lambda>0$, $f\in
L_+^1$.

We assume throughout this section that $P$ is a stochastic operator on $ L^1$,  $\varphi\colon E\to[0,\infty)$ is
a measurable function, and  that $\{S(t)\}_{t\ge0}$ is a substochastic semigroup with generator
$(A,\mathcal{D}(A))$ such that
\begin{equation}\label{eq:defA}
\mathcal{D}(A)\subseteq  L^1_{\varphi}\quad \text{and}\quad \int_{E}Af(x)\, m(dx)=-\int_{E} \varphi(x) f(x)\, m(dx)
\end{equation}
for $f\in \mathcal{D}(A)_+=\mathcal{D}(A)\cap L^1_+$, where
\[
 L^1_{\varphi}=\{f\in  L^1:\int_E\varphi(x)|f(x)|m(dx)<\infty\}.
\]
Our starting point is the following generation
result~\cite{kato54,voigt87,arlotti91,banasiak01,banasiakarlotti06} for the operator
\begin{equation}\label{eq:defC}
\mathcal{\C}f=Af+P(\varphi f)\quad \text{for}\quad f\in\mathcal{D}(A).
\end{equation}

\begin{theorem}\label{thm:kato}
There exists  a substochastic semigroup $\{P(t)\}_{t\ge 0}$ on $ L^1$ such that the generator
$(\C,\mathcal{D}(\C))$ of $\{P(t)\}_{t\ge 0}$ is an \emph{extension} of the operator in~\eqref{eq:defC}, i.e.,
\begin{equation*}\label{ext}
\mathcal{D}(A)\subseteq\mathcal{D}(\C)\quad \text{and}\quad \C f=\mathcal{\C} f\quad \text{for}\quad
f\in\mathcal{D}(A),
\end{equation*}
the generator $\C$ of $\{P(t)\}_{t\ge 0}$ is characterized by
\begin{equation}\label{eq:rp} R(\lambda,\C)f=\lim_{n\to\infty}R(\lambda,A)\sum_{k=0}^n
(P(\varphi R(\lambda,A)))^k f, \quad f\in  L^1, \lambda>0,
\end{equation}
and the semigroup $\{P(t)\}_{t\ge 0}$ is \emph{minimal}, i.e., if $\{\bar{P}(t)\}_{t\ge 0}$ is another semigroup
with generator which is an extension of $(\mathcal{\C},\mathcal{D}(A))$ then $\bar{P}(t)f\ge P(t)f$ for all $f\in
L_+^1$.

Moreover, the following are equivalent:
\begin{enumerate}[\upshape (1)]
\item  $\{P(t)\}_{t\ge 0}$ is a stochastic semigroup.
\item The generator $\C$ is the closure of the operator $(\mathcal{\C},\mathcal{D}(A))$.
\item\label{i:pert} There is $f\in  L^1_+$, $f>0$ a.e. such that for some $\lambda>0$
\begin{equation}\label{eq:css1}
\lim_{n\to\infty}\|(P(\varphi R(\lambda,A)))^nf\|=0.
\end{equation}
\end{enumerate}
\end{theorem}
\begin{remark}
Note that (see e.g.~\cite{tyran09b}) the generator of $\{P(t)\}_{t\ge 0}$ is the operator
$(\mathcal{\C},\mathcal{D}(A))$ if and only if for some $\lambda >0$
\[
\lim_{n\to\infty}\|(P(\varphi R(\lambda,A)))^n\|=0.
\]
In particular, if $\varphi$ is bounded then this condition holds.
\end{remark}

We also need the substochastic operator $K\colon  L^1\to  L^1$ defined by
\begin{equation}\label{eq:K}
Kf=\lim_{\lambda\downarrow 0} P(\varphi R(\lambda,A))f\quad \text{for } f\in L^1.
\end{equation}
It follows from \cite[Theorem~3.6]{tyran09} that $K$ is stochastic  if and only if the semigroup $\{S(t)\}_{t\ge
0}$ generated by $A$ is \emph{strongly stable}, i.e.,
\begin{equation}\label{eq:sst}
\lim_{t\to\infty}S(t)f=0\quad\text{for all } f\in  L^1.
\end{equation}
Moreover, if  $K$ is mean ergodic then the minimal semigroup $\{P(t)\}_{t\ge 0}$ from Theorem~\ref{thm:kato} is
stochastic.

We study relationships between invariant densities of the operator $K$ defined by \eqref{eq:K} and invariant densities of the minimal semigroup  $\{P(t)\}_{t\ge 0}$. Our first main result in this section is the following.

\begin{theorem}\label{thm:exunifp}
Suppose that the operator $K$ has a subinvariant density $f_*$ and let
\begin{equation}\label{eq:esd}
\overline{f}_*=\sup_{\lambda> 0} R(\lambda,A)f_*.
\end{equation}
Then $\overline{f}_*$ is subinvariant for the semigroup $\{P(t)\}_{t\ge0}$. In particular, if $\overline{f}_*\in L^1$ and the semigroup $\{P(t)\}_{t\ge 0}$ is stochastic, then it has an invariant density.
\end{theorem}
\begin{proof} Let $f_\lambda=R(\lambda,A)f_*$  for $\lambda>0$. Since $R(\lambda,A)$ is the resolvent of a
substochastic semigroup, we have $f_\lambda\ge 0$,  $f_\lambda\uparrow \overline{f}_*$, and $\overline{f}_*$ is nontrivial. From
\eqref{eq:K} it follows that $P(\varphi R(\lambda,A))f_*\le K f_*\le f_*$. We have $\mathcal{D}(A)\subseteq
\mathcal{D}(\C)$ and $\C f=Af+P(\varphi f)$ for $f\in\mathcal{D}(A)$. Hence
\begin{equation*}
\C R(\lambda,A)f=\lambda R(\lambda,A)f+P(\varphi R(\lambda,A))f-f
\end{equation*}
for every $f\in L^1$, which implies that $\C f_\lambda \le \lambda f_\lambda$ for all $\lambda>0$. The semigroup
$e^{-\mu t}P(t)$ has the generator $(\C-\mu ,\mathcal{D}(\C))$, thus
\[
f-e^{-\mu t}P(t)f=\int_0^t e^{-\mu s}P(s)(\mu-\C)f ds
\]
for all $t,\mu>0$ and  $f\in \mathcal{D}(\C)$. Since  $(\mu-\C)f_\lambda\ge (\mu-\lambda) f_\lambda\ge 0$  for
every $\mu\ge \lambda>0$, we conclude that
\[
f_\lambda-e^{-\mu t}P(t)f_\lambda\ge 0
 \]
for all $\mu\ge\lambda>0$ and $t>0$. Consequently,
 \[
P(t)f_\lambda\le e^{\mu t} f_\lambda\le e^{\mu t} \overline{f}_*,
 \]
and taking pointwise limits of both sides when $\lambda\downarrow 0$ and then $\mu\downarrow 0$ shows that $\overline{f}_*$
is subinvariant for $P(t)$. Finally, if $P(t)$ is stochastic and $\overline{f}_*\in L^1$ then $\|\overline{f}_*\|>0$ and $\overline{f}_*/\|\overline{f}_*\|$ is
an invariant density for $P(t)$.
\end{proof}

We now give a useful observation.

\begin{corollary}\label{c:pos}
If the operator $K$ has a subinvariant density $f_*$ and $f_*>0$ a.e., then the semigroup $\{P(t)\}_{t\ge 0}$ is
stochastic and $\overline{f}_*$ as defined in \eqref{eq:esd} satisfies $\overline{f}_*>0$ a.e.
\end{corollary}
\begin{proof}
Since $Kf_*\le f_*$ and $f_*>0$ a.e, the operator $K$ is mean ergodic. Thus $\{P(t)\}_{t\ge 0}$ is stochastic. We have $\overline{f}_*\ge R(\lambda,A)f_*$ for
$\lambda>0$. Since  $R(\lambda,A)$ is a positive bounded operator with dense range, we get
$R(\lambda,A)f_*>0$~a.e.
\end{proof}

\begin{remark}\label{r:pos}
Note that if $\{P(t)\}_{t\ge 0}$ has an invariant density $\tilde{f}$ with  $\tilde{f}>0$ a.e. then $\{P(t)\}_{t\ge 0}$ is
stochastic. To see this we check that condition \eqref{i:pert} of Theorem~\ref{thm:kato} holds.  By \cite[Remark 3.3]{tyran09}, we obtain that
\[
\begin{split}
\|R(1,\C)f\|=\lim_{n\to \infty}\|R(1, A)\sum_{k=0}^n (P(\varphi R(1,A)))^kf\|=\lim_{n\to\infty}(\|f\|-\|(P(\varphi R(1, A)))^{n+1}f\|)
\end{split}
\]
 for any $f\in L_+^1$. On the other hand, we have $R(1,\C)\tilde{f}=\tilde{f}$,  which shows that there is $\tilde{f}\in L^1_+$,
$\tilde{f}>0$ a.e., satisfying~\eqref{eq:css1}.
\end{remark}

\begin{remark}
The assumption in Theorem~\ref{thm:exunifp} that the  subinvariant function $f_*$ is integrable is essential, as
the following example shows \cite[Example 4.3]{kato54}. Let $E$ be the set of integers and let $m$ be the
counting measure on $E=\mathbb{Z}$ so that $ L^1=l^1(\mathbb{Z})$. Consider $Af=-\varphi f$ where $\varphi$ is a
positive function such that \[ \sum_{k\in\mathbb{Z}}\frac{1}{\varphi(k)}<\infty.
\] The semigroup generated by
$Af=-\varphi f$, $f\in  L^1_{\varphi}$, being of the form
\[
S(t)f(x)=e^{-t\varphi(x)}f(x),
\]
has the resolvent operator $R(\lambda,A)f=f/(\lambda+\varphi)$, $\lambda>0$. Let $P$ be the Frobenius-Perron
operator corresponding to $T(x)=x+1$ so that $Pf(x)=f(x-1)$. We have $K=P$ and $Kf_*=f_*$ for  $f_*\equiv 1$.
Thus $\overline{f}_*=\sup_{\lambda>0}R(\lambda,A)f_*=1/\varphi$ and $\overline{f}_*\in l^1(\mathbb{Z})$. Since the operator
\[
\mathcal{\C}f(x)=-\varphi(x)f(x)+\varphi(x-1)f(x-1),
\]
with the maximal domain $\mathcal{D}_{\mathrm{max}}=\{f\in l^1(\mathbb{Z}): \mathcal{\C}f\in l^1(\mathbb{Z})\}$
is an extension of the generator $\C$ of the semigroup $\{P(t)\}_{t\ge 0}$ (see e.g. \cite[Theorem 1.1]{kato54}),
we have $\overline{f}_*\in \mathcal{D}_{\mathrm{max}}$ and $\mathcal{\C}\overline{f}_*=0$. It follows from \cite[Example 4.3]{kato54}
that $\{P(t)\}_{t\ge 0}$ is not stochastic. Thus $\overline{f}_*\not\in \mathcal{D}(\C)$, because otherwise $\overline{f}_*$ is a
strictly positive invariant density for the semigroup $\{P(t)\}_{t\ge 0}$, implying that $\{P(t)\}_{t\ge 0}$ is
stochastic, by Remark~\ref{r:pos}.
\end{remark}

We next also discuss the problem of integrability of $\overline{f}_*$ given by~\eqref{eq:esd}.
\begin{corollary}\label{c:boundedphi} Let $\overline{f}_*$ be defined
as in~\eqref{eq:esd}. If  $0\in\rho(A)$ then $\overline{f}_*\in L^1$. In particular, if the function $\varphi$ is bounded
away from $0$ then $\overline{f}_*\in L^1$.
\end{corollary}
\begin{proof} If  $0\in\rho(A)$, then $R(0,A)=-A^{-1}$ is a bounded operator and $R(0,A)=\sup_{\lambda
>0}R(\lambda,A)$, which implies that $\overline{f}_*\in L^1$.
Suppose now that there is a positive constant $\underline{\varphi}$ such that $\varphi\ge \underline{\varphi}$.
It follows from \eqref{eq:defA} that
\[
\int_E Af(x)m(dx)\le -\underline{\varphi}\|f\|
\]
for all $ f\in \mathcal{D}(A)_+$. Thus the operator $(A+\underline{\varphi},\mathcal{D}(A))$ is the generator of
a substochastic semigroup $\{T(t)\}_{t\ge 0}$ (see e.g. \cite[Lemma 4.3]{tyran09b}). On the other hand
$T(t)=e^{\underline{\varphi} t} S(t)$ for every $t>0$, which shows that $\|S(t)f\|\le e^{-\underline{\varphi}
t}\|f\|$ for all $f\in L^1$ and $t>0$. Hence, $0\in \rho (A)$.
\end{proof}

The generator $A$ might not have a bounded inverse operator, but if the semigroup $\{S(t)\}_{t\ge 0}$ is
strongly stable, then $A$ has always a densely defined inverse operator. We next recall its definition and
properties. Let  the operator $\R\colon \mathcal{D}(\R)\to  L^1$ be defined by
\begin{equation}\label{eq:R0}
\begin{split}
\R f& = \int_{0}^\infty S(s)f\,ds:=\lim_{t\to\infty}\int_{0}^t S(s)f\,ds,\\
\mathcal{D}(\R)&=\{f\in  L^1\colon \int_{0}^\infty S(s)f\,ds \text{ exists}\}.
\end{split}
\end{equation}
The mean ergodic theorem for semigroups \cite[Chapter VIII.4]{yosida78}(see also \cite[Theorem 12]{davies82})
together with additivity of the norm in $L^1$ and the characterization~\cite[Theorem 3.1]{krengellin84} of the
range of the generator of a substochastic semigroup gives the following.
\begin{proposition}\label{l:dyn}
Let $(\R,\mathcal{D}(\R))$ be defined by \eqref{eq:R0}. Then $\rg(\R)\subseteq \mathcal{D}(A)$,  $A\R f=-f$ for
$f\in\mathcal{D}(\R)$, and
\[
\mathcal{D}(\R)\subseteq \rg(A)=\{f\in  L^1: \sup_{t\ge 0}\bigl\|\int_{0}^t S(s)f\,ds\bigr\|<\infty\},
\]
where $\rg(A)=\{Af: f\in \mathcal{D}(A)\}$ is the range of the operator $A$.

Moreover, if the semigroup $\{S(t)\}_{t\ge 0}$ is strongly stable then $\mathcal{D}(\R)$ is dense,
$\rg(A)\subseteq\mathcal{D}(\R)$, $\R Af=-f$ for $f\in \mathcal{D}(A)$, and
\[
\R f=\lim_{\lambda\downarrow 0}R(\lambda,A)f, \quad f\in \mathcal{D}(\R).
\]
\end{proposition}

We can now prove the following simple fact.

\begin{corollary}\label{c:exunifp}
Let  $(\R,\mathcal{D}(\R))$  be defined by \eqref{eq:R0}. Suppose that  $K$  is stochastic. Then $K$ is the
unique bounded extension of the densely defined operator $(P(\varphi \R),\mathcal{D}(\R))$. Moreover, if $f_*$
is an invariant density for $K$ then  $\overline{f}_*=\sup_{\lambda>0}R(\lambda,A)f_*\in L^1$ if and only if  $f_*\in
\mathcal{D}(\R)$, in which case  $\overline{f}_*=\R f_*$ and $\overline{f}_*\in\mathcal{D}(A)$.
\end{corollary}
\begin{proof}
We have   $\rg(\R)\subseteq \mathcal{D}(A)$ and $ \mathcal{D}(A)\subseteq  L^1_{\varphi}$. Let $f\in \mathcal{D}(\R)_+$.
From \eqref{eq:defA} it follows that
\[
\|\varphi \R f\|=\int \varphi(x) \R f(x)m(dx)=-\int A\R f(x)m(dx).
\]
Since $A\R f=-f$, we obtain that $\|\varphi \R f\|=\|f\|$. The multiplication operator $M_\varphi\colon  L^1_{\varphi}\to
L^1$ defined by $M_\varphi f=\varphi f$ for $f\in\mathcal{D}(M_\varphi)= L^1_{\varphi}$ is closed. Since $\R
f=\lim_{\lambda \downarrow 0}R(\lambda,A)f$ and $\R f\in L_\varphi^1$, we obtain that $\lim_{\lambda\downarrow
0} \varphi R(\lambda,A)f =\varphi \R f$. Hence, $K f=P(\varphi \R f)$ and the result follows from
Proposition~\ref{l:dyn}.
\end{proof}

We next prove a partial converse of Theorem \ref{thm:exunifp}.

\begin{theorem}\label{thm:exunifpi}
Suppose that the semigroup $\{P(t)\}_{t\ge 0}$ has a subinvariant density $\tilde{f}_*\in \mathcal{D}(\C)$. Then
$P(\varphi \tilde{f}_*)<\infty$ a.e. and $P(\varphi \tilde{f}_*)$ is subinvariant for the operator $K$. Moreover, if $\varphi
\tilde{f}_*\in L^1$ then $\tilde{f}_*\in\mathcal{D}(A)$.
\end{theorem}
\begin{proof}
Let $\lambda>0$ be fixed and let $f_0=\lambda \tilde{f}_*-\C \tilde{f}_*$. Since $e^{-\lambda t}P(t)\tilde{f}_*\le \tilde{f}_*$ for every $t>0$,
we obtain that $\C\tilde{f}_*\le \lambda \tilde{f}_*$. Thus $f_0\in L^1_+$.  Define
\[
f_n=\sum_{k=0}^n (P(\varphi R(\lambda,A)))^k f_0\quad \text{and}\quad \tilde{f}_n=R(\lambda,A)f_n,\quad n\ge 0.
\]
From \eqref{eq:rp} it follows that
\[
\lim_{n\to\infty}\tilde{f}_n=\lim_{n\to\infty}R(\lambda,A)f_n=R(\lambda,\C)(f_0)=\tilde{f}_*.
\]
We have $0\le f_n\le f_{n+1}\in L^1_+$, $n\ge 0$, and $\sup_{n}f_n<\infty$ a.e. (see e.g. \cite[Lemma
6.17]{banasiakarlotti06}). Moreover, $0\le \tilde{f}_n\le \tilde{f}_{n+1}\in \mathcal{D}(A)$, $n\ge 0$, and $\sup_{n}\tilde{f}_n=\tilde{f}_*\in
L_+^1$. Thus, we obtain that
\[
P(\varphi \tilde{f}_n)=P(\varphi R(\lambda, A))f_n=f_{n+1}-f_0\in L_+^1,
\]
which gives
\begin{equation}\label{e:in}
P(\varphi \tilde{f}_*)=\sup_{n}P(\varphi \tilde{f}_n)= \sup_nf_n-f_0.
\end{equation}
Consequently, $P(\varphi \tilde{f}_*)<\infty$ a.e.
Since $\lambda R(\lambda,A)$ is substochastic, the operator $R(\lambda,A)$ can be extended to the space of
nonnegative measurable functions by setting
\[
R(\lambda,A)f=\sup_{n}R(\lambda,A) f_n, \quad \text{if } f=\sup_{n }f_n,
\]
which implies that
\[
 R(\lambda,A)P(\varphi \tilde{f}_*)\le  R(\lambda,A)f= \tilde{f}_*.
\]
Since $\varphi R(\lambda, A)P(\varphi \tilde{f}_n)\le \varphi R(\lambda, A)P(\varphi \tilde{f}_{n+1})\in L^1_+$, we conclude
that
\[
P(\varphi R(\lambda,A))(P(\varphi \tilde{f}_*))=\sup_{n}P(\varphi R(\lambda, A)P(\varphi \tilde{f}_n))\le P(\varphi \tilde{f}_*),
\]
which gives $K(P(\varphi \tilde{f}_*))\le P(\varphi \tilde{f}_*)$ and completes the proof of the first part. Suppose now that $\varphi \tilde{f}_*\in L^1$. This implies that $P(\varphi \tilde{f}_*)\in L^1$ and that $f\in L^1$, by \eqref{e:in}. Hence,
$\tilde{f}_*=R(\lambda,A)f\in \mathcal{D}(A)$.
\end{proof}

\begin{corollary}\label{c:exunifpi}
Suppose that the semigroup $\{P(t)\}_{t\ge 0}$ has an invariant density $\tilde{f}_*$. Then $P(\varphi \tilde{f}_*)$ is
subinvariant for the operator $K$. Moreover, if $\varphi \tilde{f}_*\in L^1$ and $K$ is stochastic, then  $\|\varphi
\tilde{f}_*\|>0$ and $P(\varphi \tilde{f}_*)/\|\varphi \tilde{f}_*\|$ is an invariant density for~$K$.
\end{corollary}
\begin{proof} Recall that $\tilde{f}_*$ is a fixed point of each operator $P(t)$ if and only if $\tilde{f}_*\in \ker(\C)=\{f\in
\mathcal{D}(\C):\C f=0\}$. Thus, $\tilde{f}_*\in \mathcal{D}(\C)$ and $\C \tilde{f}_*=0$. From Theorem~\ref{thm:exunifpi} it follows
that $\tilde{f}_*\in \mathcal{D}(A)$, thus $\C \tilde{f}_*=A\tilde{f}_*+P(\varphi \tilde{f}_*)=0$. Suppose that $\|P(\varphi \tilde{f}_*)\|=0$. Then
$A\tilde{f}_*=0$, which implies that $\tilde{f}_*\in \ker(A)$. Since the operator $K$ is stochastic, condition \eqref{eq:sst} holds. Recall that $A$ is the generator of the semigroup $\{S(t)\}_{t\ge 0}$.  Thus $\ker(A)=\{0\}$ and we infer that
$\tilde{f}_*=0$, which contradicts the fact that $\|\tilde{f}_*\|=1$ and completes the proof that $f_*$ is a density. Because $K$ is
stochastic,  the subinvariant $f_*$ is invariant.
\end{proof}

We establish the following useful result when combined with Theorem~\ref{thm:pr00}.

\begin{corollary}\label{c:unique}
Assume that the operator $K$ is stochastic and uniquely mean ergodic with an invariant density $f_*$. Then the  semigroup $\{P(t)\}_{t\ge 0}$ is stochastic,   it can have at most one invariant density, and
$\varphi \tilde{f}_*\in L^1$ for any invariant density $\tilde{f}_*$. Moreover, if  $\R
f_*\in L^1$, where $\R$ is as in~\eqref{eq:R0}, then $\R f_*/\|\R f_*\|$ is the unique invariant density for the semigroup $\{P(t)\}_{t\ge 0}$.
\end{corollary}
\begin{proof}  From Theorem~\ref{thm:exunifpi} it follows that if $f$ is an invariant
density for $\{P(t)\}_{t\ge 0}$ then $P(\varphi f)<\infty$ a.e. and $K(P(\varphi f))\le P(\varphi f)$. We have
$P(\varphi f)\in L^1$, by Proposition \ref{p:exsub}, implying that $\varphi f\in L^1$. Hence, $f\in
\mathcal{D}(A)$ and $f_*=P(\varphi f)/\|\varphi f\|$ is an invariant density for $K$, by
Corollary~\ref{c:exunifpi}. 
Suppose now that the semigroup $\{P(t)\}_{t\ge 0}$ has two invariant densities
$f_1,f_2$. We have  $\C f_1=0=\C f_2$ and $\C f=Af+P(\varphi f)$ for $f\in \mathcal{D}(A)$. Since $f_*$ is the unique
invariant density for the operator  $K$, we obtain that
\[
\frac{ P(\varphi f_1)}{\|\varphi f_1\|}=\frac{P(\varphi f_2)}{\|\varphi f_2\|},
\]
which implies that
\[
\frac{ A f_1}{\|\varphi f_1\|}=\frac{A f_2}{\|\varphi f_2\|}.
\]
The operator $K$ is stochastic thus $\ker(A)=\{0\}$ by \eqref{eq:sst}. Consequently
\[
\frac{ f_1}{\|\varphi f_1\|}=\frac{ f_2}{\|\varphi f_2\|}
\]
and  $f_1=f_2$, because  $\|f_1\|=\|f_2\|=1$. The last part follows from Theorem~\ref{thm:exunifp}.
\end{proof}

\begin{remark}
Observe that if the function $\varphi$ is bounded then the assumption that $K$  is mean ergodic is not needed in
Corollary~\ref{c:unique}, since then automatically the semigroup is stochastic and $P(\varphi f)\in L^1$ for
every $f\in L^1_+$. Instead we can only assume that $K$ has a unique  invariant density $f_*$.
\end{remark}

Before we give the proof of Theorem \ref{t:cont}, we recall the relation established in \cite[Section
5.2]{tyran09} between minimal PDMPs and the minimal semigroups.
Let $\{X(t)\}_{t\ge 0}$ be the minimal PDMP on $E$ with characteristics $(\pi,\varphi,\J )$ and let $m$ be a $\sigma$-finite measure on $\E=\mathcal{B}(E)$. We assume that
$P\colon  L^1\to  L^1$ is the transition operator corresponding to $\J $
and that the semigroup $\{S(t)\}_{t\geq0}$, with generator $(A,\mathcal{D}(A))$ satisfying \eqref{eq:defA}, is such that
\begin{equation}\label{eq:T0St} \int_E
e^{-\int_{0}^{t}\varphi(\pi_rx)dr}1_B(\pi_tx)f(x)\,m(dx)=\int_B S(t)f(x)\,m(dx)
\end{equation}
for all $t\ge 0$, $f\in  L^1_+$, $ B\in\E$.
Observe that if $\varphi$ satisfies  condition~\eqref{eq:phi2} then the semigroup $\{S(t)\}_{t\ge 0}$ is
strongly stable.
The semigroup $\{P(t)\}_{t\ge 0}$ will be referred to as the \emph{minimal semigroup on $ L^1$ corresponding to}
$(\pi,\varphi,\J )$.
The following result combines Theorem 5.2 and Corollary 5.3 from~\cite{tyran09}.
\begin{theorem}[\cite{tyran09}]\label{thm:exist} Let $(t_n)$ be the
sequence of jump times and
 $t_\infty=\lim_{n\to\infty} t_n$ be the explosion time for $\{X(t)\}_{t\ge 0}$.
Then the following hold:
\begin{enumerate}[\upshape(1)]
\item The operator  $K$ as defined in \eqref{eq:K}
is the transition operator corresponding to the discrete-time Markov process $(X(t_n))_{n\ge 0}$ with stochastic
kernel
\begin{equation}\label{eq:Kkernel}
\mathcal{K}(x,B)
=\int_{0}^\infty \J (\pi_sx,B)\varphi(\pi_sx)e^{-\int_{0}^s\varphi(\pi_rx)dr}ds, \quad x\in E, B\in\B(E).
\end{equation}
\item For  any $B\in\B(E)$, a density $f$, and $t>0$
\begin{equation*}
\int_{B}P(t)f(x)m(dx)=\int_E \Px(X(t)\in B,t<t_\infty)f(x)m(dx).
\end{equation*}
\item\label{cond:nexplosive} The semigroup $\{P(t)\}_{t\geq 0}$ is stochastic if and only if
\[
m\{x\in E:\ \Px(t_{\infty}<\infty)>0\}=0.
\]
In that case if the distribution of $X(0)$ has a density $f_0$
then $X(t)$ has the density $P(t)f_0$ for all $t>0$.
\end{enumerate}
\end{theorem}

Theorem \ref{t:cont} is a direct consequence of the following result. Observe also that it follows from condition~\eqref{cond:nexplosive} of Theorem \ref{thm:exist} that the process $X$ is non-explosive.

\begin{theorem}\label{t:contsemigroup}
Let $K$ be the transition operator corresponding to the stochastic kernel given by \eqref{eq:Kkernel}.
Suppose  that $K$ has a unique invariant density $f_*$ and that $f_*>0$ a.e. Then the minimal semigroup $\{P(t)\}_{t\geq0}$ corresponding to $(\pi,\varphi,\J)$ is stochastic and it can have at most one invariant density. Moreover, if condition \eqref{eq:R0v} holds, then the semigroup $\{P(t)\}_{t\geq0}$   has a unique invariant density and it is strictly positive a.e.
\end{theorem}

\begin{proof}
Since  the stochastic operator $K$ has a unique invariant density $f_*$ and  $f_*>0$ a.e., $K$ is  uniquely mean ergodic. Thus the first assertion follows from
Corollary \ref{c:unique}.  If, moreover, condition \eqref{eq:R0v} holds then $\R f_*\in L^1$,  where $\R $ is defined by \eqref{eq:R0}, since
\[
\|\R f_*\|=\int_{0}^\infty \|S(t)f_*\|dt =\int_{0}^\infty\int_{E}e^{-\int_0^t \varphi(\pi_rx)dr} f_*(x)\,m(dx)dt=\int_E\Exx(t_1)f_*(x)m(dx).
\]
In that case $\tilde{f}_*=\R f_*/\left\|\R f_*\right\|$ is the unique invariant density for $\{P(t)\}_{t\geq 0}$. \end{proof}

We conclude this section with the following characterization of asymptotic behavior of the minimal semigroup.

\begin{corollary}\label{cor:as_stab}
Assume that  the minimal semigroup  $\{P(t)\}_{t\geq 0}$ is partially integral. Suppose  that $K$ has a unique invariant density $f_*$ and that $f_*>0$ a.e.
Then $\{P(t)\}_{t\geq 0}$ is asymptotically stable if and only if
condition \eqref{eq:R0v} holds.
\end{corollary}
\begin{proof} The semigroup $\{P(t)\}_{t\geq 0}$ is stochastic. If condition \eqref{eq:R0v} holds then Theorems~\ref{t:cont} and~\ref{thm:pr00} imply asymptotic stability. To get the converse we show that we can apply  Corollary~\ref{p:exsubct} to $\R f_*$.
Since $P$ is the transition operator corresponding to~$\J $, we obtain, by approximation,
equation~\eqref{eq:T0St}, and Fubini's theorem,
\[
\begin{split}
\int_{B} P(\varphi \R f)(x)m(dx)=\int_{E}\J (x,B)\varphi(x)\R f(x)m(dx)=\int_{E}\mathcal{K}(x,B) f(x)m(dx)
\end{split}
\]
for all $B\in\B(E)$ and $f\in D(m)$. Substituting $f=f_*$ and $B=E$ gives
\[
\int_{E}\varphi(x)\R f_*(x)m(dx)=\int_{E}f_*(x)m(dx)=1,
\]
which implies that  $\varphi(x)\R f_*(x)<\infty$ for $m$-a.e.~$x\in E$.  Hence $\supp \varphi\subseteq \{x:\R f_*(x)<\infty\}$. From Corollary~\ref{c:unique} it follows that $\varphi \tilde{f}_*\in L^1$  for any invariant density $\tilde{f}_*$ for the semigroup $\{P(t)\}_{t\geq 0}$, which, by Corollary~\ref{c:exunifpi}, implies that $m(\supp \tilde{f}_*\cap\supp \varphi)>0$. From Theorem~\ref{thm:exunifp} it follows that
$\overline{f}_*=\R f_*$ is subinvarint for the semigroup $\{P(t)\}_{t\geq 0}$. Consequently, $m(\supp \tilde{f}_*\cap\{x:\R f_*(x)<\infty\})>0$ and if the semigroup is asymptotically stable then Corollary~\ref{p:exsubct}  implies that $\R f_*\in L^1$ giving condition  \eqref{eq:R0v}.
\end{proof}

\section{Sufficient conditions for existence of a unique invariant density}\label{sec:pdmp}

Let the standing hypothesis from Introduction hold and let $L^1=L^1(E,\mathcal{B}(E),m)$, where $m$ is the Lebesgue measure on $\mathbb{R}^d$. The transition
operator $P$ corresponding to $\J $, as in \eqref{d:jumpkernel}, is of the form
\[
P f=\int_\Theta \widehat{T}_{\theta}(p_\theta f)\nu(d\theta),\quad f\in L^1,
\]
where $\widehat{T}_{\theta}$ is the Frobenius-Perron operator for $T_\theta$.
The stochastic kernel $\mathcal{K}$ in \eqref{eq:Kkernel} is given by
\begin{equation*}
\mathcal{K}(x,B)=\int_{0}^\infty\int_\Theta 1_B(T_\theta(\pi_s x))p_\theta(\pi_s x)\nu(d\theta)\varphi(\pi_s
x)e^{-\int_0^s \varphi(\pi_rx)dr}ds
\end{equation*}
for $ x\in E,B\in \B(E),$ and can be represented as
\begin{equation}\label{eq:kernel_K}
\mathcal{K}(x,B)=\int_{\Theta\times (0,\infty)}1_B(T_{(\theta,s)} (x))k_{(\theta,s)}(x)\nu(d\theta)ds,
\end{equation}
where
\begin{equation}\label{eq:T,k}
T_{(\theta,s)}(x)=T_\theta(\pi_s x)\quad\text{and}\quad k_{(\theta,s)}(x)=p_\theta(\pi_s x)\varphi(\pi_s
x)e^{-\int_0^s \varphi(\pi_rx)dr}
\end{equation}
for all $(\theta,s)\in \Theta\times (0,\infty)$, $x\in E$. The transition operator $K$ on $L^1$ corresponding to
$\mathcal{K}$ becomes
\[
Kf=\int_{\Theta\times (0,\infty)} \widehat{T}_{(\theta,s)}(k_{(\theta,s)} f)\nu(d\theta)ds,\quad f\in L^1.
\]

Given $\theta^n=(\theta_1,\ldots,\theta_n)\in \Theta^n$ and $s^n=(s_1,\ldots,s_n)\in (0,\infty)^n$ we denote by $(\theta^n,s^n)$ the sequence $(\theta^{n},s^{n}) = (\theta_{n}, s_{n}, \ldots, \theta_1,s_1)$.
We define inductively transformations $T_{(\theta^n,s^n)}$ for $n\ge 1$, by setting
\[
\begin{split}
T_{(\theta^1,s^1)}(x) &= T_{(\theta_1,s_1)}(x),\\
T_{(\theta^{n+1},s^{n+1})}(x) &= T_{(\theta_{n+1},s_{n+1})}(T_{(\theta^n,s^n)}(x)),
\end{split}
\]
and nonnegative functions $k_{(\theta^n,s^n)}$ by
\[
\begin{split}
k_{(\theta^1,s^1)}(x)&=k_{(\theta_1,s_1)}(x),
\\ k_{(\theta^{n+1},s^{n+1})}(x)&=
k_{(\theta_{n+1},s_{n+1})}(T_{(\theta^n,s^n)}(x)) k_{(\theta^n,s^{n})}(x).
\end{split}
\]
Consequently, the $n$th iterate stochastic kernel $\mathcal{K}^n$ is of the form
\[
\mathcal{K}^n (x,B) = \int_{\Theta^n \times (0,\infty)^n} 1_B(T_{(\theta^n,s^n)}(x)) k_{(\theta^n,s^n)}(x) \nu^n(d\theta^n) ds^n,
\]
where $\nu^n=\nu\times\ldots\times\nu$ denotes the product of the measure $\nu$ on $\Theta^n$.

In the rest of this section we assume that both mappings  $(\theta,x)\mapsto T_\theta (x)$ and  $(\theta,x)\mapsto p_\theta (x)$ are
continuous as well as the intensity function $\varphi$. Furthermore, for every $x\in E$ and $\theta^n\in \Theta^n$ let the transformation $s^n\mapsto T_{(\theta^n,s^n)}(x)$ be continuously differentiable and let $\frac{\partial}{\partial s^n}T_{(\theta^n,s^n)}(x)$
denote its derivative.

\begin{lemma}\label{t:meyn}
Let $x_0\in E$. Assume that there exists $(\theta^n,s^n) \in \Theta^n\times (0,\infty)^n$ such that $k_{(\theta^n,s^n)}(x_0)>0$ and the rank of $\frac{\partial}{\partial s^n}T_{(\theta^n,s^n)}(x_0)$ is equal to $d$. Then there exist a constant $c_0>0$ and open sets $U_{x_0}$, $U_{y_{0}}$ containing $x_{0}$ and $y_0=T_{(\theta^{n},s^{n})}(x_0)$, respectively, such that for all $B\in\mathcal{B}(E)$ and $x\in E$
\[
\mathcal{K}^n(x,B) \geq c_0 1_{U_{x_0}}(x) m(B\cap U_{y_0}).
\]
\end{lemma}

\begin{proof}
We adapt the proof of Lemma 6.3 in \cite{benaim12} to our situation. If the rank of $\frac{\partial}{\partial s_n}T_{(\theta^n,s^n)}(x_0)$ is equal to $d$, then we can choose $d$ variables $s_{i_1},\ldots,s_{i_d}$ from $s^n=(s_1,\ldots,s_n)$ in such a way that the derivative of the transformation $(s_{i_1},\ldots,s_{i_d})\mapsto T_{(\theta^n,s^n)}(x_0)$ is invertible. In that case, we write $u=(s_{i_1},\ldots,s_{i_d})$ and we take $v$ as the remaining coordinates of $s^n$, so that, up to the order of coordinates, we denote $s^n$ by $(u,v)$. We also write $w$ for $\theta^n$.
By assumption, there exists $(\bar{u},\bar{v},\bar{w})$ such that $k_{(\bar{w},(\bar{u},\bar{v}))}(x_0)>0$ and the rank of $\frac{\partial}{\partial(u,v)}T_{(w,(u,v))}(x_0)$ is equal to $d$ for $u=\bar{u}$, $v=\bar{v}$, $w=\bar{w}$ so, in what follows, we identify every $s^n$ with this particular choice of coordinates $u$ and $v$. Since the rank is a lower semicontinuous function, the rank of $\frac{\partial}{\partial(u,v)}T_{(w,(u,v))}(x)$ is equal to $d$ in a neighborhood of $\bar{u}$, $\bar{v}$, $\bar{w}$, $x_0$.
For $(u,v)$ we define the mapping $Q=Q_{x,w}$ by the formula
\[
Q(u,v)=(T_{(w,(u,v))}(x),v).
\]
Consequently, the determinant of $\left[\frac{\partial}{\partial(u,v)}Q\right]$ is nonzero in a neighborhood of $\bar{u}$, $\bar{v}$, $\bar{w}$, $x_0$.

We can rewrite $\mathcal{K}^n$ in the form
\[
\mathcal{K}^n(x,B)=\int_{\Theta^n\times(0,\infty)^n} 1_{B\times(0,\infty)^{n-d}}(Q(u,v))k_{(w,(u,v))}(x) \nu^n(dw) du dv
\]
for all $x\in E$ and $B\in\mathcal{B}(E)$.
Using continuity, we can find a positive constant $c$ and open sets $U_{x_0}\subset E$, $U_{\bar{u}}\subset(0,\infty)^d$, $U_{\bar{v}}\subset(0,\infty)^{n-d}$ and $U_{\bar{w}}\subset\Theta^n$ such that $k_{(w,(u,v))}(x)|\det[\frac{\partial}{\partial (u,v)}Q]|^{-1}\geq c$ for $x\in U_{x_0}$, $u\in U_{\bar{u}}$, $v\in U_{\bar{v}}$, $w\in U_{\bar{w}}$. We write $U_z$ to indicate that the point $z$ belongs to $U_z$. Moreover, for $y_0=T_{(\bar{w},(\bar{u},\bar{v}))}(x_0)$ we can find an open set $U_{y_0}\subset E$ such that $U_{y_0}\times U_{\bar{v}}\subset Q(U_{\bar{u}}\times U_{\bar{v}})$. Hence, for all $x\in U_{x_0}$ and for every set $B\in\mathcal{B}(E)$ we have
\[
\mathcal{K}^n(x,B)
\geq c\int_{U_{\bar{w}}}\int_{U_{\bar{u}} \times U_{\bar{v}}} 1_{B\times U_{\bar{v}}}(Q(u,v))\left|\det\left[\frac{\partial Q}{\partial (u,v)}\right]\right| du dv \nu^n(dw).
\]
Substituting $z_1=T_{(w,(u,v))}(x)$ and $z_2=v$ we obtain
\[
\mathcal{K}^n(x,B)
\geq c\int_{U_{\bar{w}}}\int_{Q(U_{\bar{u}} \times U_{\bar{v}})} 1_{B}(z_1)1_{U_{\bar{v}}}(z_2) dz_1 dz_2 \nu^n(dw).
\]
By the choice of the set $U_{y_0}$ we get
\[
\mathcal{K}^n(x,B)\geq c\int_{U_{\bar{w}}}\int_{U_{y_0} \times U_{\bar{v}}} 1_B(z_1)1_{U_{\bar{v}}}(z_2) dz_1 dz_2 \nu^n(dw)=c_0\int_{B}1_{U_{y_0}}(y)m(dy),
\]
where $c_0=cm_{n-d}(U_{\bar{v}})\nu^n(U_{\bar{w}})$ and $m_{n-d}(U_{\bar{v}})$ is the $n-d$ dimensional Lebesgue measure of the set $U_{\bar{v}}$ when $d<n$, and it is $1$, otherwise.
\end{proof}

To apply Lemma \ref{t:meyn} we have to calculate the rank of $\frac{\partial}{\partial s^n}T_{(\theta^n,s^n)}(x_0)$, which is the most difficult part. We next describe two possibilities how to make these calculations easier.

\begin{remark}\label{rem:small_s}
Using the continuity of derivatives with respect to $s_1,\ldots,s_n$ and taking the limit when each $s_i$ goes to zero from the right,  the limit of the derivative $\frac{\partial}{\partial s^n}T_{(\theta^n,s^n)}(x_0)$ becomes of the form
\begin{equation}\label{eq:sto0}
							\left [T'_{\theta_n}(y_{n-1})\ldots T'_{\theta_1}(y_0)g(y_0) \left|
							T'_{\theta_n}(y_{n-1})\ldots T'_{\theta_2}(y_1)g(y_1)\right|
							\cdots\left|
							T'_{\theta_n}(y_{n-1})g(y_{n-1})\right.
							\right],
\end{equation}
where $y_0=x_0$ and $y_i$ for $i=1,2,\ldots,n$ is given inductively by $y_i=T_{\theta_{i}}(y_{i-1})$.
Since the transformations $T_{\theta}$, $\theta\in\Theta$, and the mapping $g$ are explicitely defined, the rank of the matrix in \eqref{eq:sto0} can be obtained much easier then the rank of $\frac{\partial}{\partial s^n}T_{(\theta^n,s^n)}(x_0)$.
Moreover, lower semicontinuity of the rank allows us to find $s^n$ with positive coordinates.
\end{remark}

\begin{remark}\label{rem:theta}
Suppose that $\Theta$ is an open subset of $\mathbb{R}^k$ for some positive $k$ and $\nu$ is the Lebesgue measure. Assume also that transformations $(\theta^n,s^n,x)\mapsto T_{(\theta^n,s^n)}(x)$ are continuously differentiable.
Then, for a given $x\in E$ we can consider the derivative of the transformation $(\theta^n,s^n)\mapsto T_{(\theta^n,s^n)}(x)$, which can be written as
\[
\frac{\partial T_{(\theta^n,s^n)}(x)}{\partial (\theta^n,s^n)}=\left[
							\frac{\partial T_{(\theta^n,s^n)}(x)}{\partial (\theta_1,s_1)}\left|
							\frac{\partial T_{(\theta^n,s^n)}(x)}{\partial (\theta_2,s_2)}\right|
							\cdots \left|
							\frac{\partial T_{(\theta^n,s^n)}(x)}{\partial (\theta_n,s_n)}\right.
							\right].
\]
Lemma \ref{t:meyn} remains true under the assumption that the rank of the matrix  $\frac{\partial}{\partial (\theta^n,s^n)}T_{(\theta^n,s^n)}(x)$, instead of $\frac{\partial}{\partial s^n}T_{(\theta^n,s^n)}(x)$, is equal to $d$.
As in \cite{meyn91}, we can introduce the notation
\begin{equation}\label{e:XiPsi}
\begin{split}
\Xi_{n} &:= \Xi_{n}(x,(\theta^{n+1},s^{n+1})) = \left[\frac{\partial T_{(\theta,s)}(y)}{\partial y}\right]_{\substack{y=T_{(\theta^n,s^n)}(x)\;\;\\\theta=\theta_{n+1}, s=s_{n+1}}}, \\
\Psi_{n} &:= \Psi_{n}(x,(\theta^{n+1},s^{n+1})) = \left[\frac{\partial T_{(\theta,s)}(y)}{\partial (\theta,s)}\right]_{\substack{y=T_{(\theta^n,s^n)}(x)\;\;\\\theta=\theta_{n+1}, s=s_{n+1}}},
\end{split}
\end{equation}
where the derivatives are evaluated at $T_{(\theta^n,s^n)}(x)$ and for $\theta=\theta_{n+1}, s=s_{n+1}$.
Here $T_{(\theta^n,s^n)}(x) = x$ for $n=0$. Then the matrix $\frac{\partial}{\partial (\theta^n,s^n)}T_{(\theta^n,s^n)}(x)$ can be rewritten in the form
\[
\frac{\partial T_{(\theta^n,s^n)}(x)}{\partial (\theta^n,s^n)} = \left[ \Xi_{n-1} \cdots \Xi_{1}\Psi_{0} | \Xi_{n-1} \cdots \Xi_{2}\Psi_{1} | \cdots | \Xi_{n-1}\Psi_{n-2} | \Psi_{n-1} \right].
\]
\end{remark}

Now we provide sufficient conditions for which the assumptions of Theorem \ref{t:inv} are satisfied for the transition operator $K$ corresponding to $\mathcal{K}$ as defined in \eqref{eq:kernel_K}.
For each $x\in E$ we define the set
\begin{equation}\label{eq:O(x)}
\begin{split}
\mathcal{O}^+(x)=\{  T_{(\theta^n,s^n)}(x): \ & \text{the rank of}\ \frac{\partial T_{(\theta^n,s^n)}(x)}{\partial s^n}\ \text{is}\ d \text{ and}\\
														& k_{(\theta^n,s^n)}(x)>0 \text{ for } (\theta^n,s^n) \in \Theta^n\times (0,\infty)^n,\ n\geq1\}.
\end{split}
\end{equation}

\begin{corollary}\label{c:stcom}
Assume that $\mathcal{O}^+(x)\neq\emptyset$ for every $x\in E$. Suppose also that there is no $K$-absorbing sets.
Then either $K$ is sweeping with respect to compact subsets of $E$ or $K$ has a unique invariant density $f_*$. In the latter case, $f_*>0$~a.e.
\end{corollary}

\begin{remark}\label{rem:K-abs}
Observe that if there is a non-trivial $K$-absorbing set, then there is a non-trivial set $B$ such that
\[
\bigcup_{n\geq1}\bigcup_{(\theta^n,s^n)\in\Theta^n\times(0,\infty)^n} T_{(\theta^n,s^n)}(B)\subset B.
\]
This may be rewritten as
\[
\bigcup_{x\in B}\mathcal{O}(x)\subset B,
\]
where $\mathcal{O}(x)=\bigcup_{n\geq1}\mathcal{O}_n(x)$ and
\[
\mathcal{O}_n(x)=\{T_{(\theta^n,s^n)}(x):\ (\theta^n,s^n)\in\Theta^n\times(0,\infty)^n\},\quad n\ge 1.
\]
\end{remark}

Once we know that a unique invariant density exists for the operator $K$, we can use Corollary \ref{cor:as_stab} to prove asymptotic stability of the semigroup $\{P(t)\}_{t\geq0}$. We need to check that the semigroup $\{P(t)\}_{t\geq0}$ is partially integral. Our next result gives a simple condition for that.

\begin{lemma}\label{t:meyncont}
Let $x_0\in E$, $t>0$ and $n\geq1$. Define
\[
\Delta_t^n=\{s^n=(s_1,\ldots,s_n)\in (0,\infty)^n: s(n):=s_1+\ldots+s_n<t\}
\]
and assume that there exists $(\theta^n,s^n) \in \Theta^n\times\Delta_t^n$ such that $k_{(\theta^n,s^n)}(x_0)>0$ and the rank of $\frac{\partial}{\partial s^n}\pi_{t-s(n)}T_{(\theta^n,s^n)}(x_0)$ is equal to $d$. Then there exist a constant $c_0>0$ and open sets $U_{x_0}$, $U_{y_{0}}$ containing $x_{0}$ and $y_0=\pi_{t-s(n)}T_{(\theta^{n},s^{n})}(x_0)$, respectively, such that for all $B\in\mathcal{B}(E)$ and $x\in E$
\begin{equation}\label{e:meyncont}
\Px(X(t)\in B) \geq c_0 1_{U_{x_0}}(x) m(B\cap U_{y_0}).
\end{equation}
In particular, the semigroup $\{P(t)\}_{t\geq0}$ is partially integral.
\end{lemma}

\begin{proof}
Observe that if $x$ is such that
$\Px(t_\infty<\infty)=0$, then
\[
\begin{split}
\Px(X(t)\in B)&=\sum_{k=0}^\infty \Px(X(t)\in B, t_k\le t< t_{k+1}).
\end{split}
\]
Thus, to check whether condition \eqref{e:meyncont} is satisfied, it is sufficient to prove that
\begin{equation}\label{e:contproof}
\Px(\pi_{t-t_n}X(t_n)\in B, t_n\leq t< t_{n+1})\geq c_0 1_{U_{x_0}}(x) m(B\cap U_{y_0}).
\end{equation}
Since we have
\begin{equation*}
\Px(\pi_{t-t_n}X(t_n)\in B, t_n\leq t< t_{n+1})=
\int_{\Theta^n\times (0,\infty)^n} 1_{\Delta_t^n}(s^n)1_B(\pi_{t-s(n)}
T_{(\theta^n,s^n)}(x))\psi_{t-s(n)}(T_{(\theta^n,s^n)}(x))k_{(\theta^n,s^n)}(x)\nu^n(d\theta^n)ds^n,
\end{equation*}
where $\phi$ is a positive continuous function defined by
$\psi_t(x)=e^{-\int_0^t\varphi(\pi_r x)dr}$ for $x\in E$, $t\ge 0$, we can obtain \eqref{e:contproof} in an analogous way as in the proof of Lemma \ref{t:meyn}.
\end{proof}

As in Remarks \ref{rem:small_s} and \ref{rem:theta}, we can simplify the calculation of the rank of $\frac{\partial}{\partial s^n}\pi_{t-s(n)}T_{(\theta^n,s^n)}(x_0)$.

\begin{remark} \label{rem:cont}
Analogously to Remark \ref{rem:small_s}, the limit of the derivative $\frac{\partial}{\partial s^n}\pi_{t-s(n)}T_{(\theta^n,s^n)}(x_0)$ when $s_1,\ldots,s_n,t$ go to zero, is of the form
\begin{equation}\label{eq:small_s}
							\left[
							T'_{\theta_n}(y_{n-1})\ldots T'_{\theta_1}(y_0)g(y_0)\!-\!g(y_n)\left|
							\cdots\right|
							T'_{\theta_n}(y_{n-1})g(y_{n-1})\!-\!g(y_n)
							\right],
\end{equation}
where $y_0=x_0$ and $y_i=T_{\theta_{i}}(y_{i-1})$ for $i=1,2,\ldots,n$.
A similar approach to check this "rank condition" is used in \cite[Proposition 3.1]{pichorrudnicki00} and \cite{rudnickipichortyran02} as well as in \cite{bakhtinhurth12} and \cite{benaim12}.

In the case when $\Theta$ is an open subset of $\mathbb{R}^k$ and we can take derivative with respect to $\theta\in\Theta$ we have
\[
\frac{\partial \pi_{t-s(n)}T_{(\theta^n,s^n)}(x)}{\partial (\theta^n,s^n)}=\left[
							\frac{\partial \pi_{t-s(n)}T_{(\theta^n,s^n)}(x)}{\partial (\theta_1,s_1)}\right|
							\cdots \left|
							\frac{\partial \pi_{t-s(n)}T_{(\theta^n,s^n)}(x)}{\partial (\theta_n,s_n)}
							\right],
\]
for $x\in E$.
Using the notation as in \eqref{e:XiPsi} and defining additionally the derivatives
\begin{equation*}
\begin{split}
&\Upsilon_{n}:=\Upsilon_{n}(x,(\theta^n,s^n),k)=
\left[\frac{\partial\pi_{s}y}{\partial (\theta_k,s_k)}\right]_{\substack{s=t-s(n)\,\quad\\ y=T_{(\theta^n,s^n)}(x)}}
=\left[0|-g(T_{(\theta^n,s^n)}(x))\right],\\
&\Upsilon_{x,n}:=\Upsilon_{x,n}(x,(\theta^{n},s^{n}))=
\left[\frac{\partial\pi_{s}y}{\partial y}\right]_{\substack{s=t-s(n)\,\quad\\ y=T_{(\theta^n,s^n)}(x)}},
\end{split}
\end{equation*}
we have
\begin{equation}\label{eq:theta_diff}
\frac{\partial \pi_{t-s(n)}T_{(\theta^n,s^n)}(x)}{\partial (\theta^n,s^n)} =
\left[ \Upsilon_{n} + \Upsilon_{x,n} \Xi_{n-1} \cdots \Xi_{1}\Psi_{0} | \cdots | \Upsilon_{n} + \Upsilon_{x,n} \Xi_{n-1}\Psi_{n-2} | \Upsilon_{n} + \Upsilon_{x,n} \Psi_{n-1} \right].
\end{equation}
\end{remark}

We will show how our results can be applied in one particular example in the next section.
We conclude this section with the idea how to write dynamical systems with random switching as studied in \cite{bakhtinhurth12,benaim12,pichorrudnicki00}, in our framework. Given a finite or countable set $I$, consider a family of locally Lipschitz functions $g^i\colon\mathbb{R}^d\to\mathbb{R}^d$, $i\in I$, and the differential equation
\begin{equation}\label{eq:diff_eq_g_i}
\left\{
\begin{array}{l}
x'(t) = g^{i(t)}(x(t)),\\
i'(t) = 0.
\end{array}
\right.
\end{equation}
We assume that there exists a set $M\subset\mathbb{R}^d$ such that for every $i_0\in I$ and $x_0\in M$ the solution $x(t)$ of $x'(t)=g^{i_0}(x(t))$ with initial condition $x(0)=x_0$ exists and that $x(t)\in M$ for all $t\geq 0$. We denote this solution by $\pi_t^{i_0}(x_0)$.
Then, the general solution of the system \eqref{eq:diff_eq_g_i} may be written in the form
\[
\pi_t(x_0,i_0)=(\pi_t^{i_0}(x_0),i_0),\quad(x_0,i_0)\in M\times I.
\]
This gives one semiflow on $E=M\times I$ which is generated by the differential equation
\[
(x'(t),i'(t))=g(x(t),i(t)),
\]
where the function $g\colon\mathbb{R}^d\times I\to\mathbb{R}^{d+1}$ is of the form
\[
g(x,i) = (g^i(x),0),\quad x\in\mathbb{R}^d,\ i\in I.
\]
Let $m$ be the product of the Lebesgue measure $m_d$ on $\mathbb{R}^d$ and the counting measure $\nu$ on $\Theta=I$.
We define the transformation $T_{j}\colon\mathbb{R}^d\times I\to \mathbb{R}^d\times I$, $j\in I$, by
\[
T_j(x,i) = (x,j),\quad x\in\mathbb{R}^d,\ i,j\in I.
\]
Each transformation is nonsingular with respect to $m$ since
\[
m(T_j^{-1}(B\times\{i\}))=\left\{\begin{array}{ll}
																	m_d(B)\nu(\{j\}) & \text{if } i=j,\\
																	0 & \text{if } i\neq j.
																	\end{array}\right.
\]
We assume that $q_j(x,i)$, $j\neq i$, are nonnegative continuous functions satisfying $\sum_{j\neq i} q_j(x,i)<\infty$ for all $i\in I$, $x\in\mathbb{R}^d$ . Then we can define the intensity function $\varphi$ by
\[
\varphi(x,i) = \sum_{j\neq i} q_j(x,i)
\]
and the densities $p_j$, $j\in I$, by $p_i(x,i) = 0$ and
\[
p_j(x,i)=
\left\{
\begin{array}{ll}
1, & \varphi(x,i)=0,\ j\neq i,\\
\frac{q_j(x,i)}{\varphi(x,i)}, & \varphi(x,i)\neq 0, j\neq i.
\end{array}
\right.
\]
As a particular example of dynamical systems with random switching, one can consider a standard birth-death process by taking $q_{i+1}(x,i)=b_i$, $q_{i-1}(x,i)=d_i$ and $q_j(x,i)=0$ for $j<i-1$ or $j>i+1$. Then $\varphi(x,i)=b_i+d_i<\infty$.

According to \eqref{eq:T,k}, we can write explicitly formulas for the density
\[
k_{(j,s)}(x,i)=q_j(\pi_s^ix,i)e^{-\int_{0}^{s}\varphi(\pi_r^ix,i) dr}
\]
and for the transformation
\[
T_{(j,s)}(x,i)=T_j(\pi_s^ix,i)=(\pi_s^ix,j).
\]
For each $n$ we get a general form of $T_{(\theta^n,s^n)}(x_0,i_0)$ for $\theta^n=(i_1,\ldots,i_n)$ and $s^n=(s_1,\ldots,s_n)$, which is
\[
T_{(\theta^n,s^n)}(x_0,i_0)=(\pi_{s_n}^{i_{n-1}}\circ\ldots\circ\pi_{s_2}^{i_1}\circ\pi_{s_1}^{i_0}x_0,i_n).
\]
This may be rewritten as
\[
T_{(\theta^n,s^n)}(x_0,i_0)=(x_n,i_n),
\]
where
\[
x_n=\pi_{s_n}^{i_{n-1}}\circ\ldots\circ\pi_{s_2}^{i_1}\circ\pi_{s_1}^{i_0}x_0=\pi_{s_n}^{i_{n-1}}(x_{n-1}).
\]
Using this notation we adjust the definition of the set in \eqref{eq:O(x)} as follows
\begin{equation*}
\begin{split}
\mathcal{O}^+(x_0,i_0)=\{(x_n,i_n)\in E:\ & \text{the rank of}\ \frac{\partial x_n}{\partial s^n}\ \text{is}\ d\text{ and}\\
														& q_{i_n}(x_n,i_{n-1})\ldots q_{i_1}(x_0,i_1)>0 \text{ for } i_1,\ldots,i_n\in I,\ s_1,\ldots,s_n>0,\ n\geq1\}.
\end{split}
\end{equation*}
 For such semiflow with jumps, we can modify the proof of Lemma \ref{t:meyn}, to get the next result for the corresponding operator $K$.
\begin{corollary}\label{c:jumps}
Assume that $\mathcal{O}^+(x,i)\neq\emptyset$ for every $(x,i)\in E=M\times I$. Suppose also that there is no $K$-absorbing sets.
Then either $K$ is sweeping with respect to compact subsets of $E$ or $K$ has a unique invariant density $f_*$. In the latter case, $f_*>0$ a.e. In particular, if $M$ is compact, then $K$ has a unique invariant density.
\end{corollary}
To verify whether the rank of $\frac{\partial x_n}{\partial s^n}$ is equal to $d$, we may use either Remark \ref{rem:small_s} or Lie brackets as in \cite[Theorem~3]{bakhtinhurth12}, \cite[Theorem~4.4]{benaim12}.
It is worth to mention that in \cite{benaim12} it is assumed that the set $M$ is compact.

\section{A two dimensional model of gene expression with bursting}\label{sec:example}

In this section we study a particular example of a two dimensional PDMP $X(t)=(X_1(t),X_2(t))$ with values in
$E=[0,\infty)^2$. We let $X_1$ and $X_2$ denote the concentrations of mRNA and protein respectively. We assume
that the protein molecules undergo degradation at rate $\gamma_2$ and that the translation of proteins from mRNA
is at rate $\beta_2$. The mRNA molecules undergo degradation at rate $\gamma_1$ that is interrupted at random
times
\[
0<t_1<t_2<\ldots <t_n<t_{n+1}<\ldots
\]
when new molecules are being produced with intensity $\varphi$ depending at least on the current level $X_2$ of proteins. At each $t_k$ a random amount $\theta_k$ of mRNA molecules is produced, which is independent of everything else
and distributed according to a density $h$. Therefore, $p_{\theta}(x)=h(\theta)$ and the transformation $T_{\theta}$ is given by the formula
\[
T_{\theta}(x_1,x_2) = (\theta+x_1, x_2),\quad\theta\in(0,\infty).
\]
Hence, the jump kernel is  of the form
\[
\J ((x_1,x_2),B)=\int_{0}^\infty 1_{B}(\theta+x_1,x_2)h(\theta)d\theta,
\]
so that the transition operator $P$ is as follows
\[
Pf(x_1,x_2)=\int_{0}^{x_1}f(z,x_2)h(x_1-z)dz.
\]

The semiflow is defined by the solutions of the system of equations
\begin{equation*}
\frac{dx_{1}}{dt}=-\gamma_1 x_1,\quad
\frac{dx_{2}}{dt}=-\gamma_2x_2+\beta_2 x_1,
\end{equation*}
and it can be expressed by the formula
\[
\pi_t(x_1,x_2)=(x_1e^{-\gamma_1 t},x_2e^{-\gamma_2 t}+x_1\vartheta(t)),
\]
where
\begin{equation*}
\vartheta(t)=\frac{\beta_2}{\gamma_1-\gamma_2}(e^{-\gamma_2t}-e^{-\gamma_1 t}).
\end{equation*}
If $\gamma_1>\gamma_2$ then we have $\pi_t(E)\subseteq E$ for all $t\ge 0$ and the transformation $T_{(\theta,s)}$ is of the form
\[
T_{(\theta,s)}(x_1,x_2)=(\theta+x_1e^{-\gamma_1 s},x_2e^{-\gamma_2 s}+x_1\vartheta(s)).
\]
The assumption $\gamma_1>\gamma_2$ is biologically reasonable, see e.g. \cite{yvinec14} and references therein, were it was recalled that a fast process of mRNA degradation has been observed in bacterias, i.e. \textit{E. coli}.
The production of mRNA molecules can be described by exponential density with mean~$b$
\[
h(\theta)=\frac{1}{b}e^{-\theta/b},\quad \theta>0,
\]
while the intensity $\varphi$ is a Hill function depending only on the second coordinate,
\[
\varphi(x_1,x_2)=\frac{\kappa_1+\kappa_2x_2^N}{1+\kappa_3x_2^N},
\]
where $N,\kappa_1>0$ and $\kappa_2,\kappa_3\ge 0$ are constants. If  $\kappa_3=0$ we assume, additionally, that $N\leq1$ and $\gamma_2>{b\beta_2\kappa_2}/{(\gamma_1-\gamma_2)}$.
We show that the minimal semigroup $\{P(t)\}_{t\geq0}$ is asymptotically stable.

Taking $\Theta=(0,\infty)$ with $\nu$ being the Lebesgue measure on $(0,\infty)$, we can express the stochastic kernel $\mathcal{K}$ as in \eqref{eq:kernel_K}.
With the help of Corollary \ref{c:stcom} we prove that the transition operator $K$ corresponding to $\mathcal{K}$ has a unique invariant density, which is strictly positive a.e.
First, we need to check the assumptions of Corollary \ref{c:stcom}.
The function $k_{(\theta,s)}(x)$ defined as in \eqref{eq:T,k} is strictly positive for all $x\in E$ and $\theta, s>0$, since both $\varphi$ and $h$ are strictly positive.
Taking into account Remark \ref{rem:theta}, we consider the derivative $\frac{\partial}{\partial (\theta^n,s^n)}T_{(\theta^n,s^n)}(x)$ instead of $\frac{\partial}{\partial s^n}T_{(\theta^n,s^n)}(x)$.
We have
\[
\Xi_{k}=\left[
\begin{array}{cc}
e^{-\gamma_1 s_{k+1}}, & 0 \\
\vartheta(s_{k+1}), & e^{-\gamma_2 s_{k+1}}
\end{array}
\right],\quad
\Psi_{k}=\left[
\begin{array}{cc}
1, & \multirow{2}{*}{$g(\pi_{s_{k+1}}  T_{(\theta^{k},s^{k})}(x))$} \\
0, &
\end{array}
\right],
\]
where
\[
g(x)=\left(
								\begin{array}{c}
								-\gamma_1 x_1\\
								-\gamma_2 x_2 + \beta_2 x_1
								\end{array}
						\right)\quad \text{for } x=(x_1,x_2).
\]
For arbitrary $\theta_1,s_1>0$ we can calculate
\[
\frac{\partial T_{(\theta^1,s^1)}(x)}{\partial (\theta^1,s^1)} = [\Psi_{0}] = \left[
\begin{array}{cc}
1, & -\gamma_1 x_1 e^{-\gamma_1 s_1}\\
0, & -\gamma_2 x_2 e^{-\gamma_2 s_1} + x_1\frac{\beta_2}{\gamma_1-\gamma_2}(\gamma_1 e^{-\gamma_1 s_1} - \gamma_2 e^{-\gamma_2 s_1})
\end{array}
\right].
\]
The rank of $\frac{\partial}{\partial (\theta^1,s^1)}T_{(\theta^1,s^1)}(x)$ is equal to $2$ if and only if
\[
-\gamma_2 x_2 e^{-\gamma_2 s_1} + x_1\frac{\beta_2}{\gamma_1-\gamma_2} (\gamma_1 e^{-\gamma_1 s_1} - \gamma_2 e^{-\gamma_2 s_1})\neq0.
\]
If this condition does not hold we need to consider $T_{(\theta_2,s_2)}(T_{(\theta_1,s_1)}(x))$. We have
\[
\frac{\partial T_{(\theta^2,s^2)}(x)}{\partial (\theta^2,s^2)} = [\Xi_1\Psi_{0} | \Psi_1] \\
= \left[
\begin{array}{cc|cc}
e^{-\gamma_1 s_2}, & e^{-\gamma_1 s_2}g_1(\pi_{s_1}x) & 1, & g_1(\pi_{s_2} T_{(\theta^1,s^1)}(x))\\
\vartheta(s_2), & \vartheta(s_2)g_1(\pi_{s_1}x) + e^{-\gamma_2 s_2}g_2(\pi_{s_1}x) & 0, & g_1(\pi_{s_2} T_{(\theta^1,s^1)}(x))
\end{array}
\right]
\]
and, looking at the first and the third column, we see that the rank of $\frac{\partial}{\partial (\theta^2,s^2)} T_{(\theta^2,s^2)}(x)$ is equal to $2$.
This implies that $\mathcal{O}^+(x)\neq\emptyset$ for every $x\in E$.

We now show that there is no $K$-absorbing sets. By Remark \ref{rem:K-abs} it is enough to show that $(0,\infty)^2\subset\mathcal{O}(x)$ for $m$-a.e. $x\in E$.
Assume first that the point $x=(x_1,x_2)$ is such that $x_2<\beta_2x_1/\gamma_2$. Then its trajectory has the shape shown in Figure \ref{fig:semiflow1}. Then the grey area covers the set $\mathcal{O}_1(x)$ and we see that consecutive iterates give the rest.
Suppose now that $x_2>\beta_2x_1/\gamma_2$. Then the set $\mathcal{O}_1(x)$ is as in Figure \ref{fig:semiflow2}.

\begin{figure}[htb]
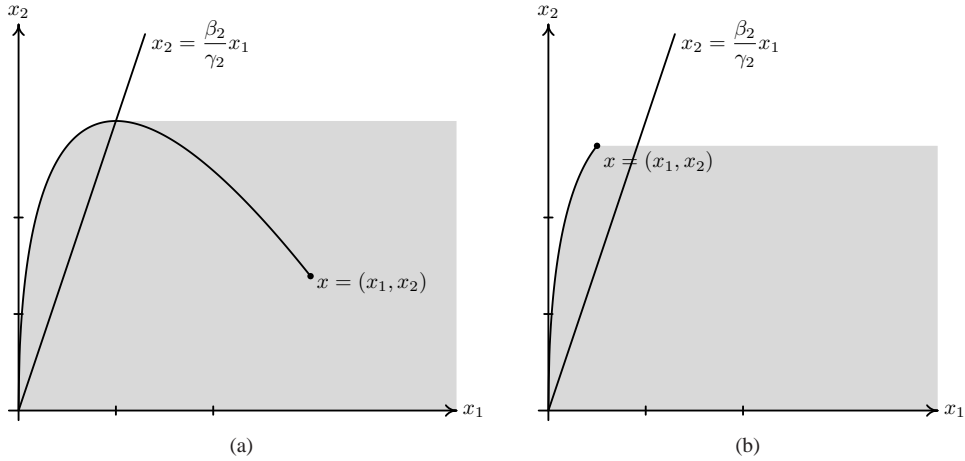

\centering
\subfigure[]{\includegraphics[scale=0.8]{portret1.mps}\label{fig:semiflow1}}
\hspace{0.5cm}
\subfigure[]{\includegraphics[scale=0.8]{portret2.mps}\label{fig:semiflow2}}
\caption{A graphical representation of the set $\mathcal{O}_1(x)$}\label{fig:semiflow}
\end{figure}

Corollary \ref{c:stcom} implies that either $K$ is sweeping with respect to compact sets or $K$ has a unique invariant density~$f_*$. To exclude sweeping, we use Proposition \ref{l:sweeping} for the operator $K$ and we take
\[
V(x)=V(x_1,x_2)=x_1\frac{\beta_2}{\gamma_1-\gamma_2}+x_2.
\]
We have
\[
V(X(t_1))-V(X(0))=\frac{\beta_2}{\gamma_1-\gamma_2}\theta_1-V(X(0))(1-e^{-\gamma_2t_1}).
\]
Since $t_1$ has the distribution function as in \eqref{eq:cdf}, we obtain
\[
\Exx(1-e^{-\gamma_2 t_1})=\gamma_2 \int_0^\infty e^{-\gamma_2 t}e^{-\int_{0}^t\varphi(\pi_s(x))ds}dt.
\]
Hence, we get
\begin{equation}\label{eq:V}
\int_EV(y)\mathcal{K}(x,dy)-V(x)=\Exx(V(X(t_1))-V(X(0)))\\
=\int_{0}^{\infty}W(t,x)e^{-\int_0^t\varphi(\pi_s(x))ds}dt,
\end{equation}
where
\[
W(t,x)=\frac{b\beta_2}{\gamma_1-\gamma_2}\varphi(\pi_tx)-V(x)\gamma_2e^{-\gamma_2t}.
\]
Notice that $W$ is bounded from above by a constant and
that $W(t,x)$ tends to $-\infty$ as $\|x\|\to\infty$ for every $t$.
Since the function $\varphi$ has a positive lower bound $\underline{\varphi}$, we obtain
\[
\int_{0}^{\infty}e^{-\int_0^t\varphi(\pi_s(x))ds} dt\leq\frac{1}{\underline{\varphi}} \quad \text{for all } x\in E.
\]
From Fatou's Lemma it follows that
\begin{equation}\label{eq:Fatou}
\limsup_{\left\|x\right\|\rightarrow\infty}\int_{0}^{\infty}W(t,x)e^{-\int_0^t\varphi(\pi_s(x))ds} dt
< 0.
\end{equation}
The function in \eqref{eq:V} is continuous, thus bounded on compact sets. Consequently, \eqref{eq:Fatou} implies that condition \eqref{cond:Lap} is satisfied and completes the proof that $K$ has a unique invariant density.

Now we look at the process $X=\{X(t)\}_{t\geq0}$. The matrices $\Upsilon_n$ and $\Upsilon_{x,n}$ from Remark \ref{rem:cont} are of the form
\begin{equation*}
\Upsilon_{n}=\left[\begin{array}{cc}
										0, & \multirow{2}{*}{$-g(T_{(\theta^{n},s^{n})}(x))$} \\
										0,
										\end{array}\right],\quad
\Upsilon_{x,n}=\left[\begin{array}{cc}
										e^{-\gamma_1(t-s(n))}, & 0 \\
										\vartheta(t-s(n)), & e^{-\gamma_2(t-s(n))}
										\end{array}\right].										
\end{equation*}
Hence $\frac{\partial}{\partial(\theta^2,s^2)}\pi_{t-s(2)}T_{(\theta^2,s^2)}(x)$ can be expressed by
\[
\begin{split}
\frac{\partial\pi_{t-s(2)}T_{(\theta^2,s^2)}(x)}{\partial(\theta^2,s^2)}&
=[\Upsilon_2+\Upsilon_{x,2}\Xi_1\Psi_0|\Upsilon_2+\Upsilon_{x,2}\Psi_1]\\
&=\left[\begin{array}{cc|cc}
e^{-\gamma_1(t-s_1)}, & * & e^{-\gamma_1(t-s(2))}, & *\\
e^{-\gamma_1s_2}\vartheta(t-s(2))+e^{-\gamma_2(t-s(2))}\vartheta(s_2), & * & \vartheta(t-s(2)), & *
\end{array}\right],
\end{split}
\]
where the first and the third column are linearly independent and the remaining columns are not important for the calculation. It is worth to notice that we need to use \eqref{eq:theta_diff} instead of the matrix in \eqref{eq:small_s} since its every two columns are linearly dependent. This proves that Lemma \ref{t:meyncont} holds, in other words, the semigroup $\{P(t)\}_{t\geq 0}$ corresponding to the process $X$ is partially integral.
We conclude from Corollary \ref{cor:as_stab} that the semigroup $\{P(t)\}_{t\ge 0}$ is asymptotically stable.


\end{document}